\documentclass{amsart}

\renewcommand{\hat}{\widehat}

\newcommand{\half}{{\textstyle\frac12}}

\newcommand{\qf}[1]{{\langle{#1}\rangle}} % formes quadratiques
\newcommand{\pf}[1]{{\langle\!\langle{#1}\rangle\!\rangle}} %formes de
\newcommand{\ba}{\overline{\rule{2.5mm}{0mm}\rule{0mm}{4pt}}} %canonical involution

\newcommand{\cf}{\mathcal{F}}
\newcommand{\cc}{\mathcal{C}}

\newcommand{\mz}{\mathbb{Z}}
\newcommand{\mq}{\mathbb{Q}}

\newcommand{\qform}[1]{{\langle{#1}\rangle}} % formes quadratiques
\newcommand{\pform}[1]{{\langle\!\langle{#1}\rangle\!\rangle}} %formes de Pfister

\newcommand{\fSB}[1]{\cf(#1)} % function field of SB

\DeclareMathOperator{\br}{Br}
\DeclareMathOperator{\End}{End}
\DeclareMathOperator{\ad}{ad}
\DeclareMathOperator{\Ad}{Ad}
\DeclareMathOperator{\ind}{ind}
\DeclareMathOperator{\Nrd}{Nrd}
\DeclareMathOperator{\Trd}{Trd}
\DeclareMathOperator{\Int}{Int}
\DeclareMathOperator{\charac}{char}
\DeclareMathOperator{\Orth}{O}

\newtheorem{lem}{Lemma}[section]
\newtheorem{prop}[lem]{Proposition}
\newtheorem{thm}[lem]{Theorem}

\newtheorem{question}{Question}
\theoremstyle{remark}
\newtheorem{remark}[lem]{Remark}
\newtheorem{remarks}[lem]{Remarks}
\newtheorem{example}[lem]{Example}

\title[Orthogonal involutions and generic splitting]{Orthogonal
  involutions on central simple algebras and function fields of
  Severi--Brauer varieties}   
\author{Anne Qu\'eguiner-Mathieu}
\address{Universit\'e Paris 13\\
Sorbonne Paris Cit\'e\\
LAGA - CNRS (UMR 7539)\\
F-93430 Villetaneuse, France}
\email{queguin@math.univ-paris13.fr}
\thanks{The first author is grateful to the second author
  and the Universit\'e catholique de Louvain for their hospitality while the work
  for this paper was carried out.}

\author{Jean-Pierre Tignol}
\address{ICTEAM Institute, Box L4.05.01\\
Universit\'e catholique de Louvain\\
B-1348 Louvain-la-Neuve, Belgium}
\email{jean-pierre.tignol@uclouvain.be}
\thanks{The second author acknowledges support
    from the Fonds de la Recherche Scientifique--FNRS under grant
    n$^\circ$~J.0149.17.}


\begin{document}

\begin{abstract}
An orthogonal involution $\sigma$ on a central simple algebra $A$,
after scalar extension to the function field $\fSB A$ of the
Severi--Brauer variety of $A$, is adjoint to a quadratic form
$q_\sigma$ over $\fSB A$, which is uniquely defined up to a scalar factor.
Some properties of the involution, such as hyperbolicity, and isotropy
up to an odd-degree extension of the base field, are encoded in this
quadratic form, meaning that they hold for the involution $\sigma$ if
and only if they hold for $q_\sigma$. As opposed to this, we prove
that there exists non-totally decomposable orthogonal involutions that
become totally decomposable over $\fSB A$, so that the associated form
$q_\sigma$ is a Pfister form. We also provide examples of
nonisomorphic involutions on an index $2$ algebra that yield similar
quadratic forms, thus proving that the form $q_\sigma$ does not
determine the isomorphism class of $\sigma$, even when the underlying algebra has index $2$. As a consequence, we show
that the $e_3$ invariant for orthogonal involutions is not classifying
in degree $12$, and does not detect totally decomposable involutions
in degree $16$, as opposed to what happens for quadratic forms.  
\end{abstract}

\maketitle

\section{Introduction}  

In characteristic different from $2$, every orthogonal involution on a
split central simple algebra is the adjoint of a nondegenerate
quadratic form. Therefore, the study of orthogonal involutions can be
thought of 
as an extension of quadratic form theory. Reversing the viewpoint, one
may try and reduce any question on involutions to a question on
quadratic forms by extending scalars to a splitting field of the
underlying algebra $A$. This is even more relevant as one may
generically split $A$ by tensoring with the function field $\fSB A$ of
its Severi--Brauer variety. Thus, to any orthogonal 
involution $\sigma$ on $A$, we may associate a quadratic form
$q_\sigma$ on $\fSB A$, which is unique up to a scalar factor, and
encodes properties of $\sigma$ over splitting fields of the algebra
$A$. 
 
Many properties of orthogonal involutions are preserved under field
extensions, hence transfer from $\sigma$ to  $\sigma_{\fSB A}$, and can be
translated into properties of $q_\sigma$. For instance, if the
involution $\sigma$ is isotropic or hyperbolic, so is the quadratic
form $q_\sigma$. Moreover, two conjugate involutions $\sigma$ and
$\sigma'$ yield similar quadratic forms $q_\sigma$ and
$q_{\sigma'}$. Conversely, some properties of involutions can be tracked
down by looking at the involution after scalar extension to
$\fSB A$, or equivalently at the associated quadratic forms over
this function field. For example, Karpenko proved that an
involution which is hyperbolic over $\fSB A$ already is hyperbolic
over the base field (see~\cite[Theorem~1.1]{Karphyporth}). It is
expected that the same holds for isotropy. Only a weaker result is
known in general for 
now, namely, an involution which is isotropic over $\fSB A$ also is
isotropic after an odd degree extension of the base field
(see~\cite[Theorem 1]{Karpisoorth}).  
In this work,
we consider the following property: an
involution $\sigma$ (or the algebra with involution $(A,\sigma)$) is
said to be \emph{totally decomposable} if
\begin{equation}
\label{eq:decomp}
(A,\sigma)\simeq(Q_1,\sigma_1) \otimes\cdots\otimes(Q_n,\sigma_n)
\end{equation}
for some quaternion algebras with involution $(Q_i,\sigma_i)$. Totally
decomposable involutions can be considered as an analogue of Pfister
forms in quadratic form theory. Indeed, by a theorem of
Becher~\cite{Becher}, 
if the algebra $A$ is split, an orthogonal involution on $A$ is totally
decomposable if and only if it is adjoint to a Pfister form, which
means that it admits a decomposition~\eqref{eq:decomp} in which each
quaternion factor is split. Therefore, every totally decomposable orthogonal
involution is adjoint to a Pfister form after generic splitting of the
underlying algebra. Whether the converse holds is a classical question,
raised in~\cite[\S\,2.4]{BPQ}:

\begin{question}
  \label{q:dec}
  Let $\sigma$ be an orthogonal involution on a central simple algebra
  $A$ of $2$-power degree. Suppose $q_\sigma$ is similar to a Pfister
  form, i.e., $\sigma_{\fSB A}$ is totally decomposable. Does it
  follow that $\sigma$ is totally decomposable ?
\end{question}

If $\deg A=4$ or $8$, cohomological invariants can be used to give a
positive answer.
In Section~\ref{Pfister.section} we give examples
showing that the answer is negative in degree~$16$ and index $4$ and
$8$. 
We even prove a slightly stronger result, namely the involutions in our examples remain non-totally decomposable over any 
odd degree extension of the base field. 
We thus disprove a conjecture of Garibaldi~\cite[Conjecture
(16.1)]{G}. See \S\ref{sec:answer} for a summary of results on
Question~\ref{q:dec}. 

The more general problem underlying questions of the type above is to
determine how much information on the involution is lost when the
algebra is generically split. In~\cite{dCQMZ}, it is proved that if
orthogonal involutions $\sigma$ and $\sigma'$ on a central simple
algebra $A$ are motivic equivalent over $\fSB A$, in an appropriate
sense which amounts to motivic equivalence of $q_\sigma$ and
$q_{\sigma'}$, then they already are motivic equivalent over the base
field. We address the analogous property for isomorphism:

\begin{question}
  \label{q:iso}
  Let $\sigma$, $\sigma'$ be orthogonal involutions on a central
  simple algebra $A$. Suppose $q_\sigma$ and $q_{\sigma'}$ are
  similar, i.e., $\sigma_{\fSB A} \simeq \sigma'_{\fSB A}$. Does it
  follow that $\sigma\simeq\sigma'$?
\end{question}

As for Question~\ref{q:dec}, cohomological
invariants yield a  positive answer for algebras of low
degree. 
The answer is also known to be positive when the algebra $A$ is
  split, hence $\fSB A$ is purely transcendental over the base field,
  see \cite[Cor.~3.2]{Hoffmann:similarity}.
Examples showing that the 
answer to Question~\ref{q:iso} is negative in degree~$8$ and index~$4$
or $8$ were provided in~\cite[\S 4]{QT:Arasondeg8}. The case of
index~$2$ is very specific however, as demonstrated by 
Becher's theorem on total decomposability (see
Proposition~\ref{dec:prop}), and because function fields of 
conics are excellent (see~\cite{PSS}). Many questions that are still
open in general are solved in index $2$, using this peculiarity. For
instance, it is known that an orthogonal involution on a central
simple algebra $A$ of index~$2$ that becomes isotropic over $\fSB A$
is isotropic over the base field, see~\cite{PSS}.
Therefore, we restrict in Section~\ref{index2.section} to the case
where the index is~$2$. We identify a few cases where the answer is
positive, but we construct examples in 
degree~$8$ and $12$ showing that the answer is negative in
general: we exhibit skew-hermitian forms of rank~$4$ or $6$ over a
quaternion algebra $Q$ that are similar after scalar extension to the
function field $\fSB Q$ of the corresponding conic, even though they
are not similar over $Q$. 
Since the group of isometries $\Orth(h)$ of
a skew-hermitian form $h$ (which is a classical group of orthogonal
type) determines the form up to similarity, we thus have
skew-hermitian forms $h$, $h'$ such that $\Orth(h)_{\fSB Q} \simeq
\Orth(h')_{\fSB Q}$ even though $\Orth(h)\not\simeq\Orth(h')$. 
By
contrast, if $h_{\fSB Q}\simeq h'_{\fSB Q}$ then $h\simeq h'$ because
scalar extension to $\fSB Q$ induces an injective map on Witt groups
of skew-hermitian forms, as shown by Dejaiffe~\cite{Dej} and
Parimala--Sridharan--Suresh~\cite[Prop.~3.3]{PSS}.
We refer to~\S\ref{sec:answer} for a summary of results on
Question~\ref{q:iso}. 

In addition, our examples have interesting consequences on
cohomological invariants of orthogonal involutions.  
It is well known that the degree $1$ and degree $2$ invariants,
respectively related to discriminants and Clifford algebras, are
defined in a similar way for quadratic forms and for involutions and
share analogous properties regarding classification
and decomposability criteria in both settings
(see~\cite{Tignol:hyderabad} for precise statements). Using the
so-called Rost invariant, which assigns a degree $3$ cohomology class
to any torsor under some absolutely almost simple simply connected
algebraic group, one may also extend the so-called Arason invariant of
quadratic form theory to orthogonal involutions, under some additional
conditions on the involution and on the underlying
algebra. Nevertheless, our examples show that the Arason
invariant does not have the 
same properties in both settings. In particular, it is not classifying
for orthogonal involutions on a degree $12$ and index $2$ algebra,
proving that Proposition~4.3 in~\cite{QT:ArasonDocumenta} is optimum
(cf. Remark~\ref{e3ind2.rem} (ii)). Moreover, it vanishes for a
$16$-dimensional quadratic form if and only if it is similar to a
Pfister form, but it may vanish for a non-totally decomposable
involution on a degree~$16$ central simple algebra (see
Remark~\ref{e3td.rem}).   

All the examples in Sections~\ref{Pfister.section} and
\ref{index2.section} have in common the use of the so-called `generic
sum' construction,  which was introduced in~\cite{QT:outer}, and
used there to construct examples of algebraic
groups without outer automorphisms. 
Questions~\ref{q:dec} and \ref{q:iso} turn out to be related in a
somewhat unexpected way: the
examples disproving Garibaldi's conjecture are generic sums of
involutions that provide a negative answer to
Question~\ref{q:iso}. To explain this relation, we develop 
a notion of `unramified algebra with involution' in
Section~\ref{Pfister.section}. 

The second-named author is grateful to Sofie Beke and Jan Van Geel for
motivating discussions on the totally decomposable case in
Theorem~\ref{thm:iso}\eqref{item:isod}.  

\subsection{Notations} 
\label{notation.section}

All fields in this paper have characteristic different from $2$; in
particular, the valued fields we consider are non-dyadic. We generally
use the notation and terminology in \cite{KMRT}, to which we refer for
background information on central simple algebras with involution. In
particular, for $n$-fold Pfister forms we write
\[
\pf{a_1,\ldots,a_n}=\qf{1,-a_1}\cdots\qf{1,-a_n}.
\]

By an algebra with involution, we always mean a central simple algebra
with an involution of the first kind. If $(D,\gamma)$ is a division
algebra with involution, and $h\colon\,M\times M\rightarrow D$ is a
nonsingular hermitian 
or skew-hermitian form with respect to $\gamma$, we let $\Ad(h)$ denote
the algebra with involution $\bigl(\End_DM,\ad(h)\bigr)$, where
$\ad(h)$ is the adjoint involution of the form $h$. In
particular, for every nonsingular
$r$-dimensional quadratic form $q$ over a field $k$, $\Ad(q)$ stands for
$\bigl(\End_k(k^r),\ad(q)\bigr)$. Recall from \cite[\S4.A]{KMRT} that
every algebra with involution can be represented as $\Ad(h)$ for some
nonsingular hermitian or skew-hermitian form. The algebra with
involution $\Ad(h)$ (or the involution $\ad(h)$) is said to be
\emph{isotropic} (resp.\ \emph{hyperbolic}) if the form $h$ is
isotropic (resp.\ hyperbolic). 

The group of \emph{similarity factors} of an algebra with involution
$(A,\sigma)$ over a field $k$ is defined as
\[
G(A,\sigma) = \{\mu\in k^\times\mid \mu=\sigma(g)g\text{ for some
  $g\in A$}\}.
\]
For a nonsingular hermitian or skew-hermitian form $h$, we write
$G(h)$ for $G\bigl(\Ad(h)\bigr)$. This group has the following
alternative description:
\[
G(h)=\{\mu\in k^\times\mid \qf\mu \cdot h \simeq h\}.
\]

\subsection{Cohomological invariants}
\label{cohinv.section}
Every $k$-algebra with orthogonal involution $(A,\sigma)$ of even
degree has a discriminant (see \cite[\S7.A]{KMRT})
\[
d_\sigma=e_1(\sigma)\in k^\times/k^{\times2},
\] 
which defines a first cohomological invariant of $\sigma$. The
discriminant also 
defines a quadratic \'etale algebra $Z=k[X]/(X^2-d_\sigma)$,
which we also call the discriminant of $\sigma$ for short. If $\sigma$
has trivial discriminant, its Clifford algebra $\cc(\sigma)$ is a
direct product of  
two central simple algebras $\cc_+(\sigma)\times \cc_-(\sigma)$ whose
Brauer classes 
satisfy $[\cc_+(\sigma)]-[\cc_-(\sigma)]=[A]$ (see~\cite[\S9.C]{KMRT} or
\cite[\S\,3.4]{Tignol:hyderabad}); 
the $e_2$ invariant of $\sigma$ has values in the quotient of the
Brauer group $\br(k)$ by the subgroup generated by the Brauer class
of~$A$; it is defined as the image of $[\cc_+(\sigma)]$ or
$[\cc_-(\sigma)]$: 
\[
e_2(\sigma)=[\cc_+(\sigma)]+\qform{[A]}=
[\cc_-(\sigma)]+\qform{[A]}\in\br(k)/\qform{[A]}. 
\] 
If $\sigma$ has trivial $e_1$ and $e_2$ invariants and the coindex of
$A$ is even (i.e., $\ind A$ divides $\half\deg A$), one may define
an Arason invariant  
\[
e_3(\sigma)\in H^3(k,\mu_4^{\otimes2})/k^\times\cdot[A],
\] 
by using the Rost invariant of Spin groups
(see~\cite[\S3.5]{Tignol:hyderabad}).
 
Since $k$ is quadratically closed in $\fSB A$, and the kernel of the
scalar extension map $\br(k)\rightarrow\br(\fSB A)$ is $\qform{[A]}$
by a theorem of Amitsur, the $e_1$ and the $e_2$ invariants of an
involution are trivial if and only if they are trivial over
  $\fSB A$, and for orthogonal involutions $\sigma$,
$\sigma'$ on $A$ we have $e_1(\sigma)=e_1(\sigma')$ (resp.\
$e_2(\sigma)=e_2(\sigma')$) if and only if $e_1(\sigma_{\fSB
  A})=e_1(\sigma'_{\fSB A})$ (resp.\ $e_2(\sigma_{\fSB
  A})=e_2(\sigma'_{\fSB A})$). The same applies to the Arason invariant
if either $A$ has index $\leq 4$, or $A$ is Brauer-equivalent to a
tensor product of three quaternion algebras, but fails in general
(see~\cite{K:triquat}, \cite{K:torsion} and \cite{Peyre}).  

When the discriminant of the orthogonal involution $\sigma$ is not
trivial (i.e., when $Z$ is a field), the Clifford algebra
$\cc(\sigma)$ is a central simple $Z$-algebra. The $k$-isomorphism
class of $\cc(\sigma)$ has the same property as the $e_1$ and $e_2$
invariants:

\begin{lem}
  \label{lem:clif}
  Let $\sigma$, $\sigma'$ be orthogonal involutions on a central
  simple $k$-algebra $A$ of even degree. If $\sigma_{\fSB A}\simeq
  \sigma'_{\fSB A}$, then $\cc(\sigma)$ and $\cc(\sigma')$ are
  isomorphic as $k$-algebras.
\end{lem}

\begin{proof}
  As observed above, the isomorphism $\sigma_{\fSB A}\simeq
  \sigma'_{\fSB A}$ implies that $e_1(\sigma)=e_1(\sigma')$, hence the
  centers $Z$ and $Z'$ of $\cc(\sigma)$ and $\cc(\sigma')$ are
  $k$-isomorphic. If $e_1(\sigma)=e_1(\sigma')=0$, then $Z\simeq
  Z'\simeq k\times k$, and
  \[
  \cc(\sigma)\simeq\cc_+(\sigma)\times\cc_-(\sigma),\qquad
  \cc(\sigma')\simeq\cc_+(\sigma')\times\cc_-(\sigma').
  \]
  From $e_2(\sigma)=e_2(\sigma')$, it follows that $\cc_+(\sigma)$ is
  isomorphic to $\cc_+(\sigma')$ or $\cc_-(\sigma')$, hence
  $\cc(\sigma)\simeq_k\cc(\sigma')$. 

  For the rest of the proof, assume
  $Z$ is a field, and choose an arbitrary isomorphism $Z\simeq Z'$ to
  identify $Z'$ with $Z$. After scalar extension to $Z$ we have
  $e_1(\sigma_Z)=e_1(\sigma'_Z)=0$, hence the first part of the proof
  yields $\cc(\sigma_Z)\simeq\cc(\sigma'_Z)$ as $Z$-algebras. But
  letting $^\iota\cc(\sigma)$ denote the conjugate $Z$-algebra of
  $\cc(\sigma)$ under the nontrivial $k$-automorphism $\iota$ of
  $Z/k$, we have
  \[
  \cc(\sigma_Z)=\cc(\sigma)\otimes_kZ\simeq
  \cc(\sigma)\times{}^\iota\cc(\sigma)  \quad\text{and likewise}\quad
  \cc(\sigma'_Z)\simeq\cc(\sigma')\times{}^\iota\cc(\sigma').
  \]
  Therefore, from $\cc(\sigma_Z)\simeq\cc(\sigma'_Z)$ it follows that
  $\cc(\sigma)\simeq\cc(\sigma')$ or $^\iota\cc(\sigma')$ as
  $Z$-algebras, hence $\cc(\sigma)\simeq\cc(\sigma')$ as $k$-algebras.
\end{proof}

\subsection{Synopsis of results on Questions~\ref{q:dec} and
  \ref{q:iso}}
\label{sec:answer}

Question~\ref{q:dec} has a positive answer if either the degree or the
index of the algebra is small enough. More precisely, we have the
following:  
\begin{prop}
  \label{dec:prop}
  Let $(A,\sigma)$ be a central simple algebra with orthogonal
  involution over a field $k$. Suppose
  $\sigma_{\fSB A}$ is totally decomposable. If in addition we have either 
   $\deg A=4$ or $8$,
 or  $\ind A=2$,
   then $\sigma$ is totally decomposable.
\end{prop}

\begin{proof}
The first assertion follows easily from the cohomological criteria of decomposability that can be established in these degrees, see~\cite[Proposition 2.10]{BPQ}. The index $2$ case is a theorem of Becher \cite[Th.~2]{Becher}.
\end{proof} 
As opposed to this, the answer is negative in general in degree $16$
and index $\geq 4$. More precisely, as a result of Theorem~\ref{thm:exgendec}
below, we obtain the following:
\begin{thm}
\label{dec:thm}
There exist central simple algebras with orthogonal involution
$(A,\sigma)$ satisfying all the following conditions:  
\begin{enumerate} 
\item[(i)] 
$\deg A=16$ and $\ind A=4$ or $8$, and
\item[(ii)] 
$\sigma_{\fSB A}$ is totally decomposable, and
\item[(iii)] $\sigma$ is not totally decomposable, and remains so over
  any odd degree extension of the base field.  
\end{enumerate}
\end{thm}

The algebra with involution $(A,\sigma)$ of Theorem~\ref{thm:exgendec}
has the form $\Ad(h)$ for some hermitian form $h$ over a division
algebra $\hat D$; it is totally decomposable over $\fSB{\hat D}$ and
satisfies~(i) and (iii). To derive Theorem~\ref{dec:thm} from
Theorem~\ref{thm:exgendec}, it suffices to observe that $\fSB A$ may
be viewed as an extension of $\fSB{\hat D}$, hence the $(A,\sigma)$ of
Theorem~\ref{thm:exgendec} is also totally decomposable over $\fSB A$.
  
\begin{remark}
\label{e3td.rem}  
Any algebra with orthogonal involution satisfying conditions (i) and (ii) above satisfies 
$
e_1(\sigma_{\fSB A})=e_2(\sigma_{\fSB A})=e_3(\sigma_{\fSB A})=0
$
because $\sigma_{\fSB A}$ is the adjoint involution of a $4$-fold
Pfister form. Since the algebra $A$ in Theorem~\ref{thm:exgendec} is
Brauer-equivalent to a tensor product of three quaternion algebras, we have in addition for those examples 
\begin{equation}
\label{eq:ezero}
e_1(\sigma)=e_2(\sigma)=e_3(\sigma)=0.
\end{equation}
Therefore, as opposed to what happens for quadratic forms and in smaller degree, the vanishing of the first cohomological invariants does not characterize totally decomposable involutions in degree $16$. 
That is, the examples in Theorem~\ref{thm:exgendec} disprove
Conjecture~(16.1) in~\cite{G}. 
Note however that if $\deg A=16$ and
$\ind A=2$, then by Proposition~\ref{dec:prop}
the involution $\sigma$ is 
totally decomposable if and only if~\eqref{eq:ezero} holds, because
\eqref{eq:ezero} is equivalent to the condition that $\sigma_{\fSB A}$
be totally decomposable.
\end{remark}

Question~\ref{q:iso} also has a positive answer if the algebra has small enough degree, or under some additional condition on the algebra and on the involution. More precisely, we have the following: 
\begin{thm}
  \label{thm:iso}
  Let $\sigma$ and $\sigma'$ be orthogonal involutions on a central
  simple algebra $A$ over a field $k$, and assume $\sigma_{\fSB
    A}\simeq\sigma'_{\fSB A}$. If in addition any of the following
  conditions holds: 
  \begin{enumerate}
\renewcommand{\theenumi}{\alph{enumi}}
  \item\label{item:isob}
  $\deg A=2$ or $4$, or
  \item\label{item:isoc}
  $\deg A=6$ and $e_1(\sigma)=0$, or
  \item\label{item:isod}
  $\ind A=2$ and $\sigma$ is totally decomposable, or
  \item\label{item:isoe}
  $\deg A\equiv2\bmod4$ and there exists a quadratic field extension
  of $k$ over which $(A,\sigma)$ is split and hyperbolic,
  \end{enumerate}
then $\sigma\simeq\sigma'$.
\end{thm}

\begin{proof}
In low degree, one may use cohomological invariants to compare involutions. More precisely, 
   if $\deg A=2$, orthogonal involutions on $A$ are
  classified by 
  their $e_1$ invariant, see \cite[(7.4)]{KMRT} or
  \cite[Th.~3.6]{Tignol:hyderabad}. If $\deg A=4$, they are classified
  by their Clifford algebra, see~\cite[(15.7)]{KMRT}. If $\deg A=6$, orthogonal involutions on $A$ with
  trivial $e_1$ 
  invariant are classified by their $e_2$ invariant (or their
  Clifford algebra),  see~\cite[(15.32)]{KMRT} or
  \cite[Th.~3.10]{Tignol:hyderabad}. With this in hand, cases
  \eqref{item:isob} and \eqref{item:isoc} 
 follow from Lemma~\ref{lem:clif}.
  
 Cases \eqref{item:isod} and \eqref{item:isoe} are proved in Proposition~\ref{totdec.prop} below.
\end{proof}

Nevertheless, Question~\ref{q:iso} has a negative answer in
general. More precisely, we may add the following to the theorem
above:  
\begin{thm}
  \label{thm:isoneg}
  There exist central simple algebras $A$ with orthogonal involutions
  $\sigma$, 
  $\sigma'$ satisfying any of the following conditions:
    \begin{enumerate}
\renewcommand{\theenumi}{\alph{enumi}}
\addtocounter{enumi}{4}
 \item\label{item:isof}
  $\deg A=8$, $\ind A=4$ or $8$, and $\sigma$ and $\sigma'$ are
  totally decomposable, or 
 \item\label{item:isog}
 $\deg A=8$,  $\ind A=2$, and $e_1(\sigma)=e_1(\sigma')=0$, or
 \item\label{item:isoh}
  $\deg A=12$, $\ind A=2$, $e_1(\sigma)=e_1(\sigma')=0$ and $e_2(\sigma)=e_2(\sigma')=0$,
 \end{enumerate}
 and such that 
\[
\sigma_{\fSB A}\simeq\sigma'_{\fSB A},\quad\mbox{and yet
}\quad\sigma\not\simeq\sigma'.
\]  
   
\end{thm}

\begin{proof}
Case \eqref{item:isof} was shown in \cite[Ex.~4.2 \& 4.3]{QT:Arasondeg8}. The other two cases are new and are explained in Remark~\ref{rem:subst}
  and Examples~\ref{1.ex} and \ref{2.ex}. 

\end{proof}

\begin{remark}
\label{e3ind2.rem}
(i) By~\eqref{item:isof} and~\eqref{item:isoh},
Theorem~\ref{thm:iso}\eqref{item:isod} does not hold anymore if we
drop one of the two assumptions.  Moreover, the example of
case~\eqref{item:isog} does not satisfy $e_2(\sigma)=e_2(\sigma')=0$;
otherwise, $\sigma$ and $\sigma'$ would be totally decomposable, and
this is impossible again by Theorem~\ref{thm:iso}~\eqref{item:isod}.

(ii) In case~\eqref{item:isoh}, the $e_3$ invariants
of $\sigma$ and $\sigma'$ are defined, and the condition $\sigma_{\fSB
  A}\simeq\sigma'_{\fSB A}$ implies $e_3(\sigma)=e_3(\sigma')$ since
$\ind A=2$. Thus, Example~\ref{2.ex} shows that the $e_3$ invariant does  
not classify orthogonal involutions with trivial $e_1$ and 
$e_2$ invariants on central simple algebras of degree~$12$ and
index~$2$ (although it is classifying if the algebra $A$ is split, and
for isotropic involutions if the algebra has index $2$ by
\cite[Prop.~4.3]{QT:ArasonDocumenta}).

(iii) By Lewis~\cite[Prop.~10]{Lewis}, in all three cases, the involutions $\sigma$ and $\sigma'$ remain non-isomorphic over any odd degree extension of the base field. 
\end{remark}

\section{Generically Pfister involutions}
\label{Pfister.section}

The goal of this section is to construct examples proving
Theorem~\ref{dec:thm}. We start with a few 
observations on central simple algebras with involution over a
complete discretely valued field, introducing a notion of
\emph{unramified} algebra with involution.

Throughout this section we let $K$ denote a field with a discrete
valuation $v\colon K\to \Gamma_K\cup\{\infty\}$. We assume $K$ is complete
for this valuation, $v(K^\times)=\Gamma_K\simeq\mz$, and the residue
field $\overline K$ has characteristic different from~$2$. Let
$(D,\gamma)$ be a 
central division algebra over $K$ with an involution of the first
kind. It is known (see for instance \cite[Th.~1.4, Cor.~1.7]{TWbook})
that $v$ 
extends to a valuation 
\[
v_D\colon D\to \Gamma_D\cup\{\infty\} \qquad\text{where
  $\Gamma_D=v_D(D^\times)=\frac1{\deg D}v(\Nrd D^\times)\subset
  \frac1{\deg D}\Gamma_K$.} 
\]
The involution $\gamma$ induces an involution $\overline\gamma$ on the
residue division algebra $\overline D$. Let $Z(\overline D)$ denote
the center of $\overline D$, which is a field extension of $\overline K$.
Since $\charac\overline K\neq2$ and $\deg D$ is a $2$-power, we have 
\begin{equation}
  \label{eq:inertsplit}
  [D:K]=[\overline D:\overline K]\cdot(\Gamma_D:\Gamma_K)
  \quad\text{and}\quad (\Gamma_D:\Gamma_K)=[Z(\overline D):\overline K],
\end{equation}
see \cite[Prop.~8.64]{TWbook}. In particular, $\Gamma_D=\Gamma_K$ if
and only if 
$\overline D$ is a central division algebra over $\overline K$; when
this condition holds, we have in addition $\deg\overline D=\deg D$. 

As observed by
Scharlau \cite[p.~208]{Sch}, there always exists a uniformizing element
$\pi_D$ for $v_D$ such that $\gamma(\pi_D)=\pi_D$, except in the
following case: 
\[
\mathbf{(*)}\qquad
\parbox{105mm}{$D$ is a quaternion algebra, $\gamma$ is the canonical
  (symplectic) involution on $D$, and $\Gamma_D=\half\Gamma_K$.}
\]
In all cases, Scharlau \cite[p.~204]{Sch} shows that every anisotropic
hermitian form $h$ 
over $(D,\gamma)$ defines an anisotropic hermitian form
$\partial_1(h)$ over $(\overline D,\overline\gamma)$, called the
\emph{first residue form} of $h$. If $(*)$ does not hold, a
\emph{second residue form} $\partial_2(h)$ is defined as the first
residue form of 
$\pi_Dh$ (which is a hermitian form with respect to the involution
$\gamma'\colon d\mapsto \pi_D\gamma(d)\pi_D^{-1}$ on $D$). As an
analogue of Springer's theorem, Scharlau proves \cite[Satz~3.6]{Sch} that
mapping $h$ to $(\partial_1(h),\partial_2(h))$ (resp.\ to
$\partial_1(h)$) yields an isomorphism of Witt groups 
\[
W(D,\gamma)\stackrel\sim\to
\begin{cases}
  W(\overline D,\overline\gamma)\oplus W(\overline
  D,\overline{\gamma'}) & \text{away from case~$(*)$,}\\
  W(\overline D,\overline\gamma) & \text{in case~$(*)$.}
\end{cases}
\]

Thus, every hermitian form $h$ over $(D,\gamma)$ has a first (and, if
$(*)$ does not hold, a second) residue form, which are either $0$ or anisotropic forms with values in $\overline D$. 
We call a hermitian form over $(D,\gamma)$
\emph{unramified} if the following two conditions hold: 
\[ 
\Gamma_D=\Gamma_K\mbox{ (i.e., $D$ is unramified over $K$), and  the
  second residue form of }h\mbox{ is }0.
\] 
% Est-ce que la première condition équivaut à dire que l'algèbre est non ramifiée? 
Otherwise the form is said to be
\emph{ramified}. Thus, in case~$(*)$ every hermitian form is
ramified. We extend this terminology to algebras with involution as
follows: $(A,\sigma)$ is said to be \emph{unramified} if there exists an unramified hermitian form $h$ such that 
$(A,\sigma)\simeq\Ad(h)$, and otherwise
$(A,\sigma)$ is called \emph{ramified}.

Alternately, one may use the 
theory of gauges developed in \cite{TWbook} to give a characterization of unramified algebras with involution which does not use representations $(A,\sigma)\simeq\Ad(h)$. Recall from \cite[Th.~2.2]{TWpaper} that if
$(A,\sigma)$ is anisotropic there is a unique map
\[
g\colon A\to(\Gamma_K\otimes_\mz\mq)\cup\{\infty\}
\]
with the following properties:
\begin{enumerate}
\item[(i)]
$g(a)=\infty$ if and only if $a=0$;
\item[(ii)]
$g(a+b)\geq\min(g(a),g(b))$ for $a$, $b\in A$;
\item[(iii)]
$g(a\lambda)=g(a)+v(\lambda)$ for $a\in A$ and $\lambda\in F$;
\item[(iv)]
$g(1)=0$ and $g(ab)\geq g(a)+g(b)$ for $a$, $b\in A$;
\item[(v)]
$g(\sigma(a)a)=2g(a)$ for $a\in A$.
\end{enumerate}
Using $g$, one defines a $\overline K$-algebra $A_0$ as follows:
\[
A_0=\frac{\{a\in A\mid g(a)\geq0\}}{\{a\in A\mid g(a)>0\}}.
\]
The algebra $A_0$ is semisimple because $g$ is a $v$-gauge in the sense of
\cite[\S3.2.2]{TWbook}, see \cite[Th.~2.2]{TWpaper}. We call $g$ the
\emph{$\sigma$-special gauge} on $A$. It satisfies
$g\bigl(\sigma(a)\bigr)=g(a)$ for all $a\in A$, and is actually
characterized by this property among all the $v$-gauges on $A$.
% Est-ce l'unique jauge qui commute avec sigma, ou faut-il une hypothèse supplémentaire pour cela? 

\begin{prop}
  \label{prop:caracunram}
  The anisotropic algebra with involution $(A,\sigma)$ is unramified
  if and only if $A_0$ is a central simple $\overline K$-algebra. When
  this condition holds, we have $\deg A_0=\deg A$.
\end{prop}

\begin{proof}
  Suppose $(A,\sigma)=(\End_DM,\ad(h))$ for some hermitian space
  $(M,h)$ over a division algebra with involution
  $(D,\gamma)$. Inspection of the proof of \cite[Th.~2.2]{TWpaper} shows
  that the map $g$ is defined as follows: for $a\in\End_DM$, 
  \[
  g(a) = \min\{\alpha(a(m))-\alpha(m)\mid m\in M,\;m\neq0\},
  \]
  where, for all $m\in M$, 
  \[
  \alpha(m)=\half v_D\bigl(h(m,m)\bigr)\in\half\Gamma_D\cup\{\infty\}.
  \]

  Thus, $g=\End(\alpha)$ in the notation of \cite[p.~104]{TWbook}, and
  the set 
  \[
  \Gamma_M=\{\alpha(m)\mid m\in M\setminus\{0\}\}
  \] 
  is a union
  of cosets of $\Gamma_D$. Let $r$ be the number of cosets of
  $\Gamma_D$ in $\Gamma_M$. Since $\Gamma_D\simeq\mz$ and
  $\Gamma_M\subset \half\Gamma_D$, we have $r=1$ or
  $2$. By~\cite[Prop.~3.34 and Prop.~2.41]{TWbook} the algebra $A_0$
  is semisimple, and its center is a product of $r$ copies of
  $Z(\overline D)$. Therefore, $A_0$ is central simple over $\overline
  K$ if and only if $r=1$ and $Z(\overline D)=\overline K$. This last
  condition implies that $D$ is unramified, in particular $(*)$ does
  not hold.

For the rest of the proof, assume $D$ is unramified, and choose a
  uniformizing element $\pi_D$ for $v_D$ in $K$. Since $D$ is
  unramified, it follows that $Z(\overline D)=\overline K$, and we
  need to prove that $(A,\sigma)$ is unramified if and only if $r=1$.
  
  Note that the map $\alpha$ is the norm on $M$ used by Scharlau in
  his definition of the first residue $\partial_1(h)$;
  see~\cite[Prop.~3.1]{Sch}. If $(A,\sigma)$ is unramified we may
  assume $h$ is unramified. Since $\partial_2(h)=0$ it follows that
  $v(h(m,m))\in2\Gamma_D$ for all nonzero $m\in M$, hence $r=1$ and
  $A_0$ is central simple over $\overline K$. Conversely, if $r=1$
  then either $\Gamma_M=\Gamma_D$ or
  $\Gamma_M=\half v_D(\pi_D)+\Gamma_D$. In the first case $h$ is
  unramified, hence $(A,\sigma)$ is unramified. In the second case
  $\pi_Dh$ is unramified. Since $\ad(h)=\ad({\pi_D h})$ we may
  substitute $\pi_D h$ for $h$ and thus again conclude that
  $(A,\sigma)$ is unramified.
\end{proof}

We now establish a sufficient condition for a totally decomposable
algebra to be unramified. This provides a tool for proving that some
ramified algebras with involution are not totally decomposable.
 
\begin{prop}
  \label{prop:unramdec}
  A totally decomposable algebra with involution $(A,\sigma)$ is
  unramified if $v\bigl(G(A,\sigma)\bigr)\subset2\Gamma_K$.
\end{prop}

\begin{proof}
  We first consider the case where $A$ is a quaternion algebra. If $A$
  is a split quaternion algebra, then $A\simeq\Ad(q)$ for some binary
  quadratic form $q$ over $K$. The determinant of $q$ is a similarity
  factor of $q$, hence the condition
  $v\bigl(G(A,\sigma)\bigr)\subset2\Gamma_K$ implies that $\Ad(q)$ is
  unramified. Now, assume $A$ is a quaternion division algebra. The
$\sigma$-special gauge $g$ on $A$ does not depend on $\sigma$ and
coincides with the valuation $v_A$ 
  extending $v$, hence $A_0$ is the residue division algebra $\overline
  A$. If $\sigma$ is the canonical involution, then
  $G(A,\sigma)=\Nrd(A^\times)$, hence the condition
  $v\bigl(G(A,\sigma)\bigr)\subset2\Gamma_K$ implies $\Gamma_A=\Gamma_K$. It
  then follows by~\eqref{eq:inertsplit} that $A_0$ is a central simple
  $\overline K$-algebra, hence $(A,\sigma)$ is unramified. If $\sigma$
  is an orthogonal involution, then as observed above $A$ contains a
  uniformizing element $\pi_A$ such that $\sigma(\pi_A)=\pi_A$
  by~\cite[p.~208]{Sch}. Let $\overline\pi_A=\Trd(\pi_A)-\pi_A\in A$ be the
  conjugate quaternion. Suppose $\Gamma_A\neq\Gamma_K$. Then
  $v(\pi_A+\overline\pi_A)\in\Gamma_K$ whereas
  $v(\pi_A)=v(\overline\pi_A)\notin\Gamma_K$, hence
  $v(\pi_A+\overline\pi_A)>v(\pi_A)$, and therefore
  \[
  v(\pi_A-\overline\pi_A)=
  v(2\pi_A-(\pi_A+\overline\pi_A))=v(\pi_A)\notin\Gamma_K. 
  \]
  Now,
  \[
  \sigma(\pi_A-\overline\pi_A)\cdot(\pi_A-\overline\pi_A) =
  (\pi_A-\overline\pi_A)^2\in K^\times,
  \]
  hence $(\pi_A-\overline\pi_A)^2\in G(A,\sigma)$. But
  $v((\pi_A-\overline\pi_A)^2)=2v(\pi_A)\notin2\Gamma_K$, in contradiction
  with the hypothesis that
  $v\bigl(G(A,\sigma)\bigr)\subset2\Gamma_K$. Therefore
  $\Gamma_A=\Gamma_K$ and 
  it follows as in the previous case that $(A,\sigma)$ is
  unramified. We have thus proved the proposition in the case where
  $A$ is a single quaternion algebra.

  Now, let
  $(A,\sigma)=(Q_1,\sigma_1)\otimes_K\cdots\otimes_K(Q_n,\sigma_n)$,
  where each $Q_i$ is a quaternion $K$-algebra, and assume
  $v\bigl(G(A,\sigma)\bigr)\subset2\Gamma_K$.
   Then $(A,\sigma)$ is not
  hyperbolic because otherwise $G(A,\sigma)=K^\times$. It follows by
  \cite[Cor.~3.2]{BQ} that $(A,\sigma)$ is anisotropic, hence each
  $(Q_i,\sigma_i)$ 
  is anisotropic. Each $Q_i$ then carries a $\sigma_i$-special gauge
  $g_i$, and since $G(Q_i,\sigma_i)\subset G(A,\sigma)$ the first part
  of the proof shows that each $(Q_i,\sigma_i)$ is unramified. The
  tensor product $g=g_1\otimes\cdots\otimes g_n$ is a $v$-gauge on $A$
  by \cite[Prop.~3.41]{TWbook}, and it satisfies $g\circ\sigma=g$ by
  \cite[Prop.~1.3]{TWpaper} because $g_i\circ\sigma_i=g_i$ for all $i$
  by the uniqueness property of the $\sigma_i$-special gauge. The map
  $g$ is therefore the $\sigma$-special gauge 
  on $A$ by \cite[Th.~2.2]{TWpaper}. By \cite[Prop.~3.41]{TWbook} or
  \cite[Prop.~1.3]{TWpaper} we have 
  \[
  A_0=(Q_1)_0\otimes_{\overline K}\cdots\otimes_{\overline K}(Q_n)_0.
  \]
  Since each $(Q_i,\sigma_i)$ is unramified, it follows that each
  $(Q_i)_0$ is a central quaternion $\overline K$-algebra, hence $A_0$
  is a central simple $\overline K$-algebra. Therefore,
  Proposition~\ref{prop:caracunram} shows that $(A,\sigma)$ is unramified.
\end{proof}

Let us now use Proposition~\ref{prop:unramdec} to construct examples of
algebras with involution that are not totally decomposable, and that
remain non-totally decomposable after odd-degree scalar
extensions. Let $(D,\gamma)$ be a central division algebra with an
involution of the first kind over an arbitrary field $k$ (of
characteristic different from~$2$), and let $h_1$, $h_2$ be
nonsingular hermitian forms over $(D,\gamma)$. Consider the field
$F=k((t))$ of Laurent series in one indeterminate over $k$. Extending
scalars to $F$, we obtain the central division algebra with involution
$(\hat D,\hat\gamma)=(D,\gamma)_F$ over $F$ and the hermitian forms
$\hat h_1=(h_1)_F$ and $\hat h_2=(h_2)_F$ over $(\hat
D,\hat\gamma)$. Consider the following hermitian form over $(\hat
D,\hat\gamma)$ (cf.~\cite[\S 3.2]{QT:outer}):
\[
h=\hat h_1\perp \qf{t}\hat h_2.
\]

\begin{prop}
  \label{prop:gencons}
  Let $K$ be an odd-degree field extension of $F$, and let $v\colon
  K\to \Gamma_K\cup\{\infty\}$ denote the discrete valuation on $K$
  extending the $t$-adic valuation on $F$. If $h_1$ and $h_2$ are not
  hyperbolic, then the algebra with involution $\Ad(h)_K$ is
  ramified. If $h_1$ and $h_2$ are not similar, then
  $v\bigl(G(h_K)\bigr)\subset2\Gamma_K$. 
  If $h_1$ and $h_2$ are not hyperbolic and are not similar, then $\Ad(h)_K$ is not totally
  decomposable. 
\end{prop}

\begin{proof}
  Let $e=(v(K^\times):v(F^\times))$ be the ramification index of
  $K/F$, and let $\pi\in K$ be a uniformizing element of $K$, hence
  also of $\hat D_K=D\otimes_kK$. We have $\pi^e=ut$ for some $u\in
  K^\times$ with $v(u)=0$, and $t\equiv u\pi\bmod K^{\times2}$ since
  $e$ is odd. Therefore,
  \[
  h_K\simeq(\hat h_1)_K\perp\qf{u\pi}(\hat h_2)_K,
  \]
  hence
  \[
  \partial_1(h_K)=(h_1)_{\overline K}
  \quad\text{and}\quad \partial_2(h_K) = \qf{\overline
    u}(h_2)_{\overline K} \quad\text{in the Witt group
    $W(D_{\overline K},\gamma_{\overline K})$.}
  \]
  If $h_1$ and $h_2$ are not hyperbolic, then $(h_1)_{\overline K}$
  and $(h_2)_{\overline K}$ are not hyperbolic because $[\overline
  K:k]$ is odd, hence no scalar multiple of $h_K$ is
  unramified. Therefore, $\Ad(h)_K$ is ramified.

  Now, let $\mu\in G(h_K)$. If $v(\mu)\notin2\Gamma_K$, then there
  exists $\mu_0\in K^\times$ such that $v(\mu_0)=0$ and
  $\mu\equiv\mu_0\pi\bmod K^{\times2}$. Then from $\qf\mu h_K\simeq
  h_K$ it follows that
  \[
  \qf{\mu_0u}(\hat h_2)_K\perp \qf{\mu_0\pi}(\hat h_1)_K \simeq (\hat
  h_1)_K \perp \qf{u\pi}(\hat h_2)_K.
  \]
  Comparing the residues of each side, we obtain
  \[
  \qf{\mu_0u}(h_2)_{\overline K}= (h_1)_{\overline K}\qquad\text{in
    $W(D_{\overline K}, \gamma_{\overline K})$.}
  \]
  Since $[\overline K:k]$ is odd, a transfer argument shows that $h_1$
  and $h_2$ are similar; see~\cite[Prop. 10]{Lewis}.

  Hence if the hermitian forms $h_1$ and $h_2$ are not hyperbolic and not similar, $\Ad(h)_K$ is ramified and
  $v\bigl(G(\Ad(h)_K)\bigr)\subset2\Gamma_K$. It follows by
  Proposition~\ref{prop:unramdec} that $\Ad(h)_K$ is not totally
  decomposable. 
\end{proof}

To obtain particularly significant instances of this construction,
recall from~\cite[Ex.~4.2 \& 4.3]{QT:Arasondeg8} that there exist 
central simple $k$-algebras $A$ of degree~$8$ with orthogonal involutions
$\sigma_1$, $\sigma_2$ such that
\begin{enumerate}
\item[(i)]
$\sigma_1$ and $\sigma_2$ are not hyperbolic and are not isomorphic,
and
\item[(ii)]
over the function field $\fSB A$ of the Severi--Brauer variety of $A$,
there is a $3$-fold Pfister form $\varphi$ such that
\[
(\sigma_1)_{\fSB A}\simeq(\sigma_2)_{\fSB A}\simeq\ad(\varphi).
\]
\end{enumerate}
In these examples, the field $k$ has characteristic zero, the index
of $A$ is $4$ or $8$, and $A$ is a tensor product of three quaternion
algebras. Now, choose a division algebra $D$ 
Brauer-equivalent to $A$ and an orthogonal involution $\gamma$ on
$D$. We may then find hermitian forms $h_1$, $h_2$ over $(D,\gamma)$
such that $\sigma_1\simeq\ad(h_1)$ and $\sigma_2\simeq\ad(h_2)$. With
this choice of $h_1$ and $h_2$, the construction preceding
Proposition~\ref{prop:gencons} yields a hermitian form $h=\hat
h_1\perp \qf t \hat h_2$ over the division algebra with involution
$(\hat D,\hat\gamma)=(D,\gamma)_F$, where $F=k((t))$.

\begin{thm}
  \label{thm:exgendec}
  With the notation above, $\Ad(h)$ is a central simple $F$-algebra of
  degree~$16$ with orthogonal involution that is not totally
  decomposable over $F$ nor over any odd-degree extension of $F$. Yet,
  over the function field $\fSB{\hat D}$ of the Severi--Brauer variety
  of $\hat D$, the algebra with involution $\Ad(h)_{\fSB{\hat D}}$ is
  totally decomposable.
\end{thm}

\begin{proof}
  Since $(A,\sigma_1)_F\simeq\Ad(\hat h_1)$ and
  $(A,\sigma_2)_F\simeq\Ad(\hat h_2)$ are central simple algebras of
  degree~$8$ with orthogonal involutions, it is clear that $\Ad(h)$ is
  a central simple algebra of degree~$16$ with orthogonal
  involution. Condition~(i) on $\sigma_1$ and $\sigma_2$ implies that
  $h_1$ and $h_2$ are not hyperbolic and are not similar. Therefore,
  it follows from Proposition~\ref{prop:gencons}
 that $\Ad(h)$ is not
  totally decomposable over $F$ nor over any odd-degree extension of
  $F$. To prove the last statement, note that $\fSB A$ is a purely
  transcendental extension of $\fSB D$, hence for the Pfister form
  $\varphi$ in condition~(ii) we may choose a form defined over
  $\fSB D$. Choosing a minimal left ideal in the split algebra
  $D_{\fSB D}$,
  we set up a Morita equivalence between hermitian forms over
  $(D_{\fSB D},\gamma_{\fSB D})$ and quadratic forms over $\fSB D$, hence also
  between hermitian forms over $(\hat D_{\fSB{\hat D}},
  \hat\gamma_{\fSB{\hat D}})$ and quadratic forms over $\fSB{\hat
    D}$, which yield identifications
  $W(D_{\fSB D},\gamma_{\fSB D})=W(\fSB D)$ and $W(\hat D_{\fSB{\hat
      D}},\hat\gamma_{\fSB{\hat D}})=W(\fSB{\hat D})$. Condition~(ii)
  shows that there are $\alpha_1$, $\alpha_2\in 
  \fSB{D}^\times$ such that
  \[
  (h_1)_{\fSB D}=\qf{\alpha_1}\varphi \quad\text{and}\quad
  (h_2)_{\fSB D}=\qf{\alpha_2}\varphi \quad\text{in $W(\fSB D)$,}
  \]
  hence
  \[
  h_{\fSB{\hat D}}=\qf{\alpha_1,\alpha_2t}\varphi_{\fSB{\hat D}}
  \quad\text{in $W(\fSB{\hat D})$.}
  \]
  Since the right side is a multiple of a Pfister form, it follows
  that $\Ad(h)_{\fSB{\hat D}}$ is totally decomposable.
\end{proof}

\begin{remarks}
\label{Arason1.rem}

\noindent(1) 
The algebra $\Ad(h)$ in Theorem~\ref{thm:exgendec} has index~$4$ or
$8$. By contrast, Becher's theorem~\cite[Th.~2]{Becher} shows that if
a central simple algebra with orthogonal involution of index at
most~$2$ is totally decomposable after generic splitting, then it is
totally decomposable (see
Proposition~\ref{dec:prop}). Therefore, the 
construction used in 
Theorem~\ref{thm:exgendec} is impossible if the index of $A$ is $2$,
which means that there are no orthogonal involutions
$\sigma_1$, $\sigma_2$ satisfying conditions~(i) and (ii) if the index
of $A$ is~$2$. In \S\,\ref{index2.section}, we provide a direct proof
of this fact (see Proposition~\ref{totdec.prop}).  
\smallbreak

\noindent(2)
Let $L$ be an arbitrary field extension of the center $F$ of the
algebra $\Ad(h)$ of Theorem~\ref{thm:exgendec}. If $\Ad(h)_L$ is
isotropic, then it is also isotropic after generic splitting of $\hat
D_L$. But since $\Ad(h)$ is totally decomposable over $\fSB{\hat D}$ it
is also totally decomposable over $\fSB{\hat D_L}$, hence
$\Ad(h)_{\fSB{\hat D_L}}$ is hyperbolic, and it follows
from~\cite[Theorem~1.1]{Karphyporth} that $\Ad(h)_L$ is
hyperbolic. Thus, $\Ad(h)$ satisfies the `Isotropy $\Rightarrow$
Hyperbolicity' condition introduced in~\cite{BPQ}, and
Theorem~\ref{thm:exgendec} shows that this condition is necessary but
not sufficient for an algebra with involution to be totally
decomposable.  
\end{remarks}

\section{Generic splitting in index $2$}
\label{index2.section} 

In this section, we address
Question~\ref{q:iso}, considering only algebras with orthogonal
involution that are split and hyperbolic over a quadratic extension of
$k$ (and are therefore of index at most~$2$). Since the answer is known in the split case, 
we assume from now on that $Q$ is a central quaternion division algebra over a field $k$ (of characteristic different from~$2$).

\begin{lem}
\label{dec.lem}
Let $\sigma$ be an orthogonal involution on $A=M_r(Q)$. If there
exists a quadratic field extension $Z/k$ over which $(A,\sigma)$ is
split and hyperbolic, then $Q$ admits an orthogonal involution $\rho$
of discriminant $Z$, and $(A,\sigma)$ decomposes as  
\[
(A,\sigma)\simeq (Q,\rho)\otimes\Ad(\varphi)
\] 
for some $r$-dimensional quadratic form $\varphi$ over $k$. 
\end{lem}

\begin{proof}
Let $Z=k(\sqrt a)$ be such that $A_Z$ is split and $\sigma_Z$
  is hyperbolic. By~\cite[Th.~3.3]{BST}, there exists $g\in A$ such
  that $\sigma(g)=-g$ and $g^2=a$, hence $k(g)\simeq Z$. The
  centralizer of $g$ in $A$ is Brauer-equivalent to $A_Z$, hence it is split.
Therefore,
by~\cite[Appendix]{QT:daca}, there exists a pure quaternion $i\in Q$
such that $i^2=a$, and some elements $\alpha_i\in k^\times$ for $i=1$, \ldots,
$r$, such that $\sigma$ is adjoint to the skew-hermitian form
over $(Q,\ba)$ defined by $h=\qform{i\alpha_1,\dots,i\alpha_r}$, where
$\ba$ is the canonical involution of $Q$. The lemma follows, with
$\rho=\Int(i)\circ\ba$ and $\varphi=\qform{\alpha_1,\dots,\alpha_r}$.  
\end{proof} 

\begin{remark}
  \label{rem:subst}
  For $A=M_r(Q)$, the field $\fSB A$ is a purely transcendental
  extension of $\fSB Q$, hence for involutions $\sigma$ and $\sigma'$
  on $A$ we have $\sigma_{\fSB A}\simeq\sigma'_{\fSB A}$ if and only
  if $\sigma_{\fSB Q}\simeq\sigma'_{\fSB Q}$
  by~\cite[Lemma~4.1]{QT:Arasondeg8}. Therefore, to establish
  Theorem~\ref{thm:isoneg}\eqref{item:isog} and \eqref{item:isoh} we may
  (and will) substitute $\fSB Q$ for $\fSB A$.
\end{remark}

\begin{lem}
\label{both.lem}
Let $(A,\sigma)=(Q,\rho)\otimes\Ad(\varphi)$ be as in the previous
lemma. Any orthogonal involution $\sigma'$ on $A=M_r(Q)$ such that  
$\sigma'_{\fSB Q}\simeq \sigma_{\fSB Q}$ also decomposes as
$\sigma'\simeq\rho\otimes\ad({\varphi'})$ for some $r$-dimensional
quadratic form $\varphi'$ over $k$.  
\end{lem}

\begin{proof}
Let $Z$ be the quadratic extension of $k$ defined by the discriminant
of $\rho$, and denote by $Z_Q= Z\cdot\fSB Q$ the corresponding
quadratic extension of $\fSB Q$.  
Since $Q_Z$ is split, $Z_Q$ is a purely transcendental extension of
$Z$. Therefore, $\sigma_{\fSB Q}\simeq \sigma'_{\fSB Q}$ implies
$\sigma_{Z_Q}\simeq \sigma'_{Z_Q}$, which in turn implies
$\sigma_Z\simeq \sigma'_{Z}$ by~\cite[Lemma~4.1]{QT:Arasondeg8}. So
$(A,\sigma')$ also is split and hyperbolic over $Z$ and the previous
lemma finishes the proof.  
\end{proof} 

With this in hand, Question~\ref{q:iso} for involutions that become
hyperbolic over a quadratic splitting field of $Q$ boils down to a
quadratic form question over $\fSB Q$, as the next proposition shows:

\begin{prop} 
\label{criterion.prop}
Let $\rho$ be an orthogonal involution of $Q$ of discriminant $k(\sqrt
a)$.  
Given two $r$-dimensional quadratic forms $\varphi$ and $\varphi'$
over $k$, consider the orthogonal involutions
$\sigma=\rho\otimes\ad(\varphi)$ and
$\sigma'=\rho\otimes\ad({\varphi'})$ of $A=M_r(Q)$.  
\begin{enumerate}
\renewcommand{\theenumi}{\roman{enumi}}
\item
\label{isogen}
$\sigma'_{\fSB Q}\simeq\sigma_{\fSB Q}$ if and only if there exists
$\lambda\in \fSB Q^\times$ such that 
\[
\pform{a}\varphi'_{\fSB Q}\simeq\qform{\lambda}\pform{a}\varphi_{\fSB Q}.
\] 
\item 
\label{isobase}
$\sigma'\simeq\sigma$ if and only if there exists $\nu\in k^\times$
such that 
\[
\pform{a}\varphi'_{\fSB Q}\simeq\qform{\nu}\pform{a}\varphi_{\fSB Q}.
\] 
\end{enumerate} 
\end{prop}

\begin{proof}
Since orthogonal involutions on a quaternion algebra are classified by
their discriminant~\cite[(7.4)]{KMRT}, and $k$ is quadratically closed
in $\fSB Q$, there exists an isomorphism 
\[
(Q,\rho)_{\fSB Q}\simeq \Ad({\pform{a}_{\fSB Q}}).
\] 
Hence, for every quadratic form $\psi$ over $k$, we have 
\[
\bigl((Q,\rho)\otimes \ad(\psi)\bigr)_{\fSB Q}\simeq
\Ad({{\pform{a}\psi}_{\fSB Q}}).
\]
In particular, $\sigma_{\fSB Q}$ and $\sigma'_{\fSB Q}$ are
respectively adjoint to the quadratic forms  
$\pform{a}\varphi_{\fSB Q}$ and $\pform{a}\varphi'_{\fSB
  Q}$. Assertion~(\ref{isogen}) follows immediately.  

Since $Q$ is a division algebra, $k(\sqrt a)$ is a
field, and $\rho=\Int(i)\circ\ba$ for some pure quaternion $i\in Q$ such that $i^2\equiv a\bmod k^{\times2}$. 
The involution
$\sigma$ (resp.\ $\sigma'$) is adjoint to the skew-hermitian form
$\qform i \varphi$ (resp.\ $\qform i \varphi'$) over
$(Q,\ba)$. Therefore, $\sigma$ and $\sigma'$ are isomorphic if and
only if there exists $\nu\in k^\times$ such that
\begin{equation}
  \label{eq:simil}
  \qform i \varphi' \simeq \qform{\nu i}\varphi.
\end{equation}
Since the scalar extension map $W^-(Q,\ba)\to W^-(Q_{\fSB Q},\ba)$ is
injective (see \cite{Dej} or \cite[Prop.~3.3]{PSS}), and since
$W^-(Q_{\fSB Q},\ba)\simeq W(\fSB Q)$ by a Morita-equivalence that
carries $\qform i_{\fSB Q}$ to $\pform a_{\fSB Q}$, the 
existence of the isomorphism~\eqref{eq:simil} is equivalent to
\[
\pform a \varphi'_{\fSB Q} \simeq \qform\nu\pform a \varphi_{\fSB Q}.
\qedhere
\]
\end{proof} 

We use Proposition~\ref{criterion.prop} to prove \eqref{item:isod} and
\eqref{item:isoe} in Theorem~\ref{thm:iso}.

\begin{prop}
\label{totdec.prop}
Let $A=M_r(Q)$ be endowed with two orthogonal involutions $\sigma$ and
$\sigma'$.  
We assume that either $(A,\sigma)$ is totally decomposable, or $r$ is
odd and there exists a quadratic extension $Z/k$ over which
$(A,\sigma)$ is split and hyperbolic.   
Under any of those two conditions, if $\sigma_{\fSB Q}$ and
$\sigma'_{\fSB Q}$ are isomorphic, then $\sigma\simeq \sigma'$. 
\end{prop}

\begin{proof}
By~\cite[Th.~2]{Becher}, if $(A,\sigma)$ is totally decomposable, it
admits a decomposition  
\[
(A,\sigma)\simeq (Q,\rho)\otimes\ad(\pi),
\] 
for some orthogonal involution $\rho$ of $Q$ and some Pfister
quadratic form $\pi$ over $k$.  
Therefore in both cases there exists a quadratic extension $Z/k$ over
which $(A,\sigma)$ is split and hyperbolic\,: $Z$ is the discriminant
of $\rho$ in the totally decomposable case. In view of
Lemmas~\ref{dec.lem} and~\ref{both.lem}, we may thus assume that
$\sigma=\rho\otimes\ad(\varphi)$ and
$\sigma'=\rho\otimes\ad(\varphi')$ for some quadratic forms $\varphi$
and $\varphi'$ over $k$, and apply Proposition~\ref{criterion.prop}.  

Assume first that $\sigma$ is totally decomposable. Then we may assume
$\varphi=\pi$ is a Pfister form, and modifying it by a scalar if
necessary, we may also assume $\varphi'$ represents $1$. If
$\sigma_{\fSB Q}$ and $\sigma'_{\fSB Q}$ are isomorphic,
assertion~(\ref{isogen}) in Proposition~\ref{criterion.prop} says
that 
\[
\pform{a}\pi_{\fSB Q}\simeq\qform{\lambda}\pform{a}\varphi'_{\fSB Q}
\]
for some $\lambda\in \fSB Q^\times$. But a quadratic form that is
similar to a Pfister form and represents $1$ actually is isomorphic to
this Pfister form. Therefore, we may take $\lambda=1$, so that the
equivalent conditions of assertion~(\ref{isobase}) of
Proposition~\ref{criterion.prop} hold, with 
$\nu=1$. This concludes the proof in this case.  

Assume now that $r$ is odd. Since the quadratic forms $\varphi$ and
$\varphi'$ are $r$-dimensional, and only defined up to a scalar
factor, we may assume they have trivial discriminant. If $\sigma_{\fSB
  Q}$ and $\sigma'_{\fSB Q}$ are isomorphic, assertion~(\ref{isogen})
of Proposition~\ref{criterion.prop}
says that there exists $\lambda\in \fSB Q^\times$ such that  
\[
\pform{a}\varphi'_{\fSB Q}\simeq\qform{\lambda}\pform{a}\varphi_{\fSB
  Q}.
\]
Computing the Clifford invariant of both forms as in~\cite[V.3]{Lam},
and taking into account the fact that $r$ is odd and $\varphi$ and
$\varphi'$ have trivial discriminant, we get
$(\lambda,a)=0\in\br_2(k)$. Therefore, $\lambda$ is represented by
$\pform{a}$, so that 
\[
\qform{\lambda}\pform{a}\simeq\pform{a}.
\]  
Hence, again we may take $\lambda=1$, and the equivalent conditions of
assertion~(\ref{isobase}) of Proposition~\ref{criterion.prop} hold.  
\end{proof}

In the rest of this paper, we use Proposition~\ref{criterion.prop} to produce examples of pairs of orthogonal involutions for which the answer to Question~\ref{q:iso} is negative. 
For this, we will exhibit quadratic forms $\varphi$ and $\varphi'$ defined over $k$ such that 
\[
\pform{a}\varphi'_{\fSB Q}\simeq\qform{\lambda}\pform{a}\varphi_{\fSB
  Q}
\] 
for some $\lambda\in \fSB Q^\times$, but not for any $\lambda\in k^\times$. 
The heuristic idea behind this construction is the
following. Suppose $\varphi_1$, $\varphi_1'$, $\varphi_2$,
$\varphi_2'$ are quadratic forms over $k$ satisfying
assertion~(\ref{isogen}) of Proposition~\ref{criterion.prop} with the
same factor $\lambda\in\fSB{Q}^\times$:
\[
(\pform{a}\varphi'_1)_{\fSB Q}\simeq\qform{\lambda}(\pform{a}\varphi_1)_{\fSB
  Q}
\quad\text{and}\quad
(\pform{a}\varphi'_2)_{\fSB Q}\simeq\qform{\lambda}(\pform{a}\varphi_2)_{\fSB
  Q}.
\]
Then for every $t\in k^\times$ we also have
\[
(\pform{a}\varphi'_1\perp\qform t \pform a \varphi'_2)_{\fSB Q} \simeq
\qform{\lambda}(\pform{a}\varphi_1\perp\qform t\pform a \varphi_2)_{\fSB Q}.
\]
If $\varphi_1$, $\varphi'_1$ and $\varphi_2$, $\varphi'_2$ do not
yield a negative answer to Question~\ref{q:iso}, then we may find
$\nu_1$, $\nu_2\in k^\times$ such that
\[
(\pform a \varphi'_1)_{\fSB Q} \simeq (\qform{\nu_1}\pform a
\varphi_1)_{\fSB Q} 
\quad\text{and}\quad
(\pform a \varphi'_2)_{\fSB Q} \simeq (\qform{\nu_2}\pform a
\varphi_2)_{\fSB Q}.
\]
But if the $\nu_1$ and $\nu_2$ satisfying these equations are
sufficiently distinct, and if $t$ is `generic', there may not be any
$\nu\in k^\times$ such that
\[
(\pform{a}\varphi'_1\perp\qform t \pform a \varphi'_2)_{\fSB Q} \simeq
\qform{\nu}(\pform{a}\varphi_1\perp\qform t\pform a \varphi_2)_{\fSB Q}.
\]

\begin{example}
\label{1.ex}
Let $k_0$ be a field of characteristic different from $2$ and
$k=k_0((a))((t))$ the iterated Laurent series field in two variables
over $k_0$. Pick $b\in k_0^\times$ and let $Q=(a,b)_k$. Assume $b$ is
not a square in $k_0$; then $Q$ is a
division algebra, which admits an orthogonal involution $\rho$ of
discriminant $a$. The function field $\fSB Q$ may be represented as 
$k(X,Y)$ where $X$ and $Y$ satisfy  
\begin{equation}
\label{eq:rep}
X^2-aY^2+ab=0.
\end{equation}
Fix an element $c\in k_0^\times$ and let
\[
\varphi=\pform{b+1}\perp\qform{t}\pform{b+c^2},
\quad\mbox{and}\quad
\varphi'=\pform{b+1}\perp\qform{ct}\pform{b+c^2}.
\]
Write $D_K(\theta)$ for the set of represented values of a quadratic form
$\theta$ over a field $K$. Note that
\[
  (\sqrt b +c)^2-(b+c^2)=2c\sqrt b \quad\text{and}\quad
  \left(\frac{\sqrt b +1}{2\sqrt b}\right)^2 - (b+1)
  \left(\frac1{2\sqrt b}\right)^2=\frac1{2\sqrt b},
\]
hence $(2\sqrt b)^{-1}\in D_{k_0(\sqrt b)}(\pf{b+1})$ and $2c\sqrt
b\in D_{k_0(\sqrt b)}(\pf{b+c^2})$, and therefore
\begin{equation}
\label{eq:th5}
  c=(2\sqrt b)^{-1}(2c\sqrt b)\in k_0^\times\cap \bigl[D_{k_0(\sqrt
    b)}(\pf{b+1})\cdot D_{k_0(\sqrt 
    b)}(\pf{b+c^2})\bigr].
\end{equation}
The following proposition yields a more precise information when
$\rho\otimes\ad(\varphi)$ and $\rho\otimes\ad(\varphi')$ are
isomorphic:

\begin{prop}
  \label{prop:ex1}
  The involutions $\sigma=\rho\otimes\ad(\varphi)$ and
$\sigma'=\rho\otimes\ad(\varphi')$ satisfy $\sigma_{\fSB Q}\simeq
\sigma'_{\fSB Q}$. If $\sigma\simeq\sigma'$, then 
\begin{equation}
  \label{eq:th4}
  c\in \bigl(k_0^\times\cap D_{k_0(\sqrt
    b)}(\pf{b+1})\bigr) \cdot \bigl(k_0^\times\cap D_{k_0(\sqrt
    b)}(\pf{b+c^2})\bigr).  
\end{equation}
\end{prop}

\begin{proof}
To establish the first claim, observe that 
$\pform{a,b+1}_{\fSB Q}$ represents  
\[
X^2-a(Y+1)^2+a(b+1)=-2aY,
\] 
while $\pform{a,b+c^2}_{\fSB Q}$ represents 
\[
X^2-a(Y+c)^2+a(b+c^2)=-2aYc.
\]
Therefore $-2aY$ (respectively $-2aYc$) is a similarity factor of
$\pform{a,b+1}_{\fSB Q}$ (respectively $\pform{a,b+c^2}_{\fSB Q})$, and we
get 
\[
\pform{a,b+1}\simeq\qform{-2aY}\pform{a,b+1} \quad\text{and}\quad
\qform c \pform{a,b+c^2} \simeq \qform{-2aY}\pform{a,b+c^2},
\]
hence
\[
\pform{a}\varphi'_{\fSB Q}\simeq \qform{-2aY}\pform{a}\varphi_{\fSB
  Q}.
\]
By Proposition~\ref{criterion.prop}, this implies $\sigma_{\fSB
  Q}\simeq \sigma'_{\fSB Q}$.  

To prove the second claim, assume
$\sigma\simeq\sigma'$. Proposition~\ref{criterion.prop}\eqref{isobase}
yields $\nu\in k^\times$ such that the forms 
$\pform{a}\varphi'_{\fSB Q}$ and $\qform{\nu}\pform{a}\varphi_{\fSB
  Q}$ are isomorphic. Since the kernel of the  
restriction map $W(k)\rightarrow W(\fSB Q)$ is $\pform{a,b}W(k)$ (see
\cite[Ch.~X, Cor.~4.28]{Lam}), it follows that
\begin{equation}
  \label{eq:th}
  \pf{\nu,a,b+1} + \qf{ct,-\nu t}\pf{a,b+c^2}\in \pf{a,b}W(k).
\end{equation}

Since $k$ is the field of iterated Laurent series in $a$ and $t$ over
$k_0$, every 
square class in $k$ is represented by an element of the form $\nu_0$,
$a\nu_0$, $t\nu_0$ or $at\nu_0$ for some $\nu_0\in k_0^\times$; see
\cite[Ch.~VI, Cor.~1.3]{Lam}. If $\nu_1=-a\nu$, then
$\pform{\nu_1,a}\simeq\pform{\nu,a}$ and $\qform{-\nu_1t}\pform a
\simeq \qform{-\nu t}\pform a$. Therefore, substituting $-a\nu$ for
$\nu$ if $\nu=a\nu_0$ or $at\nu_0$, we may assume $\nu=\nu_0$ or
$t\nu_0$ with $\nu_0\in k_0^\times$. We consider these two cases
separately.
\smallbreak

\noindent\textbf{Case 1:} Suppose $\nu=t\nu_0$ with $\nu_0\in
k_0^\times$. Taking the first residue of~\eqref{eq:th} for the
$t$-adic valuation, and the first residue of the
resulting relation for the $a$-adic valuation, we obtain
\[
\pf{b+1}+\qf{-\nu_1}\pf{b+c^2}\in \pf b W(k_0).
\]
Comparing discriminants yields $(b+1)(b+c^2)\in
k_0^{\times2}\cup(bk_0^{\times2})$, hence $b+1\equiv b+c^2\bmod
k_0(\sqrt b)^{\times2}$. Then~\eqref{eq:th5} yields \eqref{eq:th4}.
\smallbreak

\noindent\textbf{Case 2:} Suppose $\nu\in k_0^\times$. Taking the
residues of~\eqref{eq:th} for the $t$-adic valuation, we obtain
\[
\pform{\nu,a,b+1}\in\pform{a,b}W\bigl(k_0((a))\bigr)
\quad\text{and}\quad
\qform{c,-\nu}\pform{a,b+c^2}\in\pform{a,b}W\bigl(k_0((a))\bigr).
\]
We next take the first residue for the $a$-adic valuation, and get
\[
\pf{\nu,b+1}\in\pf b W(k_0) \quad\text{and}\quad
\qf{c,-\nu}\pf{b+c^2}\in \pf b W(k_0).
\]
It follows that the forms $\pf{\nu,b+1}$ and $\pf{\nu c,b+c^2}$ become
hyperbolic over $k_0(\sqrt b)$. This means that $\nu$ is
represented by the form $\pf{b+1}$ over $k_0(\sqrt b)$, and $\nu c$ by
the form $\pf{b+c^2}$ over $k_0(\sqrt b)$, so the equation
$c=\nu^{-1}(\nu c)$ yields~\eqref{eq:th4}.
\end{proof}

To conclude, it remains to find fields $k_0$ and elements $b$, $c\in
k_0^\times$ such that \eqref{eq:th4} does not hold. Quadratic
extensions for which the `Common Value Property' investigated in
\cite{STW} (see also \cite[\S6]{HT}) fails yield such examples:
\begin{itemize}
\item 
we may take $k_0=\mathbb{Q}(b)$, where $b$ is an indeterminate, and
$c=2$: see \cite[Rem.~5.4]{STW};
\item
we may take $k_0=\ell(b,c)$, where $b$, $c$ are independent
indeterminates over an arbitrary field $\ell$ of characteristic~$0$: see
\cite[Rem.~5.10]{STW}. 
\end{itemize}
\end{example}

\begin{example}
\label{2.ex}
It is easy to modify Example~\ref{1.ex} to obtain nonisomorphic
orthogonal involutions on $M_6(Q)$ with trivial discriminant and
trivial $e_2$-invariant that are isomorphic after generic splitting.
From~\eqref{eq:rep}, derive
\[
X^2-a(Y+1)^2=-2aY-a(b+1)
\]
and
\[
-(b+1)(b+c^2)\left(\left(\frac X{b+c^2}\right)^2 -
  a\left(\frac{Y+c}{b+c^2}\right)^2\right) = -(b+1)\left(
  \frac{-2acY}{b+c^2} - a\right).
\]
Summing these two equations, we see that $\pf{a,(b+1)(b+c^2)}$
represents
\[
-2aY-a(b+1)+(b+1)\left(\frac{2acY}{b+c^2}+a\right) =
-2aYc',\qquad\text{with }c'=1-\frac{(b+1)c}{b+c^2}.
\]
If $c'=0$, then $b+c^2=(b+1)c$ and $\pf{a,(b+1)(b+c^2)}=\pf{a,c}$ is
hyperbolic. Since $a$ is an indeterminate over $k_0$, it follows that
$c\in k_0^{\times2}$, hence~\eqref{eq:th4} trivially holds. This is a
contradiction; therefore $c'\neq0$ and
\[
\qform{-2aYc'}\pform{a,(b+1)(b+c^2)}\simeq \pform{a,(b+1)(b+c^2)}.
\]  
Let $k_1=k((u))$, where $u$ is a new indeterminate, and
$Q_1=Q\otimes_kk_1$. 
Consider the following quadratic forms over $k_1$:
\begin{align*}
\psi& =\pform{b+1}\perp\qform{t}\pform{b+c^2}\perp
      \qform{u}\pform{(b+1)(b+c^2)},\\
\psi'& =\pform{b+1}\perp\qform{ct}\pform{b+c^2}\perp
       \qform{c'u}\pform{(b+1)(b+c^2)}
\end{align*}
and the involutions $\tau=\rho\otimes\ad(\psi)$ and
$\tau'=\rho\otimes\ad(\psi')$ over
$M_6(Q_1)$. Proposition~\ref{criterion.prop}\eqref{isogen} shows that they
satisfy  
\[
\tau_{\fSB{Q_1}}\simeq \tau'_{\fSB{Q_1}},
\]
because
\[
\pform{a}\psi'_{\fSB{Q}}\simeq \qform{-2aY}\pform{a}\psi_{\fSB{Q}}.
\]

It remains to prove that $\tau$ and $\tau'$ are not isomorphic. Using
again Proposition~\ref{criterion.prop}\eqref{isobase}, we need to
prove that there is no $\nu\in k_1^\times$ such that
\begin{multline*}
\pf{\nu,a,b+1}+\qf{ct,-\nu t}\pf{a,b+c^2}+\qf{c'u,-\nu
  u}\pf{a,(b+1)(b+c^2)}\\
 \in \pf{a,b} W\bigl(k((u))\bigr).
\end{multline*}
We may assume $\nu\in k^\times$ or $\nu=u\nu_1$ with $\nu_1\in
k^\times$. If $\nu\in k^\times$, taking the first residue for the
$u$-adic valuation yields \eqref{eq:th}, which has no solution. If
$\nu=u\nu_1$ with $\nu_1\in k^\times$, taking the first residue
yields
\begin{equation}
\label{eq:th7}
\pf{a,b+1}+\qf{ct}\pf{a,b+c^2}+\qf{-\nu_1}\pf{a,(b+1)(b+c^2)} \in
\pf{a,b}W(k).
\end{equation}
Recall $k=k_0((a))((t))$, so every 
square class in $k$ is represented by an element of the form $\nu_0$,
$a\nu_0$, $t\nu_0$ or $at\nu_0$ for some $\nu_0\in
k_0^\times$. Substituting $-a\nu_1$ for
$\nu_1$ if $\nu_1=a\nu_0$ or $at\nu_0$, we may assume $\nu_1=\nu_0$ or
$t\nu_0$ with $\nu_0\in k_0^\times$. If $\nu_1=t\nu_0$, then taking
the first residue of~\eqref{eq:th7} for the $t$-adic valuation and for
the $a$-adic valuation in the resulting relation, we obtain
$\pform{b+1}\in\pform b W(k_0)$. This implies $b+1\in k_0(\sqrt
b)^{\times2}$, hence \eqref{eq:th4} holds, a contradiction. If
$\nu_1\in k_0^\times$, then we take the second residue
of~\eqref{eq:th7} for the $t$-adic valuation and the first residue for
the $a$-adic valuation, and find $\qform c\pform{b+c^2}\in\pform b
W(k_0)$, hence $b+c^2\in k_0(\sqrt b)^{\times2}$, which also yields a
contradiction. Therefore $\tau$ and $\tau'$ are not isomorphic.  
\end{example}

\bibliographystyle{plain}
\bibliography{QTPfister}

\end{document}